%% file: main.tex
\documentclass{article}
\input{MetaDonnees}

\newcommand{\ev}{\operatorname{ev}}
\newcommand{\Gal}{\operatorname{Gal}}
\newcommand{\Sym}{\operatorname{Sym}}
\newcommand{\Jmoins}{J[2]\setminus\{ 0 \}}
\newcommand{\moins}{\setminus\{ 0 \}}

\title{A trick to ensure positive Mordell-Weil rank}
\author{Thibaut Misme\thanks{\href{mailto:mismet@tcd.ie}{mismet@tcd.ie}}}
\affil{\scriptsize{Trinity College Dublin}}

\begin{document}

\maketitle

\renewcommand{\contentsname}{Contents}
\tableofcontents                

\newpage

\section*{Abstract}

In this short note, we present a trick to ensure that the Jacobian of a given smooth curve over a number field has strictly positive Mordell-Weil rank. More explicitly, we prove that a smooth curve with no rational non-trivial 2-torsion and no rational theta characteristic has non-zero Mordell-Weil rank assuming the existence of a rational degree 1 divisor class. In particular, it implies that a generic nice curve of genus $g>1$ with a rational degree 1 class has strictly positive rank. This criteria is both of theoretical and computational interest as we show how to use it in practice. We also give refinements, including an equivalent for families of curves, and explicit examples.

\section*{Introduction}

In general it is not easy to ensure that the Jacobian $J$ of a given smooth curve $C$ of genus $g$ over a number field $\K$ has strictly positive Mordell-Weil rank. Doing so usually requires one to exhibit a non-torsion rational point on $J$; a task that may turn out to be tricky in practice.

%\revoir{Detail technique générale.}
First one has to compute an upper bound $B$ of the size $T$ of the $\K$-rational torsion of the Jacobian by computing the cardinal of the $\F_\mathbf{p}$-point $J(\F_\mathbf{p})$ for (several) prime(s) $\mathbf{p}$. Taking the greatest common divisor, one deduces a multiple $B$ of $T$. See \cite{Pooneng2}[section 5.]. It remains to pick a rational point $P\in J(\K)$ and to check whether $B\cdot P$ is trivial. The drawbacks of this strategy is that computing in the Jacobian (see \cite{Pooneng2}[section 4.] or see the Makdisi's algorithm \cite{Makdisi1}, \cite{Makdisi2}) is not that straightforward. In particular, one is more or less bound to work over $\Q$. This leads to the second problem which is that, even when $\K=\Q$, one should consider $P\in J(\Q)$ as a degree 0 linear combination of rational point on $X$. The sample of possible points $P$ to attempt is thus quite limited in practice.\\

The problem of certifying positive rank for elliptic curves or producing such curves is much easier. This is due for instance to the relations that link Selmer groups of quadratic twists and the practicability of 2-descent in genus $g=1$. Similarly, some methods exist for hyperelliptic curves (see \cite{positiveCH}). Such kind of problems are at the heart of the recent proof that Hilbert's 10th problem is undecidable for rings of integers of number fields. It has first been suggested in \cite{Poonen10Hilbert} that the existence of an elliptic curve $E$ over $\Q$ satisfying $\rank_\Q(E)=\rank_\K(E)=1$ for a number field $\K$ would deduce the undecidability of Hilbert 10th problem over the ring of integers $\mathcal{O}_\K$ of $\K$ from the equivalent known classical problem over $\Z$. The proof of very closely statements indeed led to solve Hilbert 10th problem, see \cite{Hilbert10} and \cite{Hilbert10bis}. Note nevertheless that the real difficulty is not to produce elliptic curves of positive rank but to control the exact (positive) rank of elliptic curves over some (quadratic) extensions. Similarly, for a number field $K$, Koymans and Morgan construct in \cite{rank1Jac} some twists of Jacobians of hyperelliptic curves defined over $K$ that are of rank exactly 1 over $K$.\\

In the present article, we present a way to ensure that the Jacobian of a general curve of genus $g>1$ possesses an infinite order rational point without the need of performing computation on any specific point. The technique (see Proposition \ref{main}) does however construct and exhibit an explicit point that is proved to be of infinite order under the required assumptions. It relies on the non-existence of rational 2-torsion points and of rational Theta Characteristics. Asking for a more precise knowledge of the Galois action over the 2-torsion allows us to bypass the condition on Theta Characteristic in Theorem \ref{maincor}.
%We also show (see \ref{maincor}) that a more precise knowledge of the Galois action over the 2-torsion allows to bypass the condition on Theta Characteristics.
Moreover, in both versions, we give implemented and practicable methods for general curves to certify that the hypothesis are satisfied.\\

%Motivation
For instance, the techniques can serve as an easy way to ensure that the rank of $J$ is exactly 1 if we have already bounded the rank from above by 1 by some descent method. This is facilitated by the fact that 2-descent computes relevant information for our techniques.\\
Studying the distribution of Mordell-Weil ranks has been an extensive work in arithmetic statistics. It includes producing infinite families with positive rank for elliptic curves, see for instance \cite{PythagoreCE} or \cite{CEfamily}, and for hyperelliptic curves, see \cite{barbulescu}. We extend our results to families of curves, generating an easy way to produce families of general curves with infinitely many fibers with positive rank. An explicit example is given.\\

Here is the upshot of the techniques that we are about to present (see Theorem \ref{maincor}): assuming the existence of a rational degree 1 divisor class on $C$,
$$g>1 \text{ and }  G_\K \text{ transitive on } J[2]\setminus\{0\} \implies \rank_\K(J) \geq 1. $$
In particular, it applies to a generic nice curve of genus $g>1$ with a rational degree 1 class since the Galois group of its 2-torsion is then isomorphic to the symplectic group $\Sp$. Such a curve is therefore of strictly positive rank.\\

We organised the present article as following. In the first section we outline the main Proposition \ref{main} and the (small) theoretical background on Theta Characteristics it demands. We also expose a local equivalent and a link with Ceresa Cycles. In the second section we decline Proposition \ref{main} to the version that does not imply Theta Characteristics. We then tackle the computational practicability of our results. In section 4, we extend Theorem \ref{maincor} to families of curves. To conclude, we give explicit examples in the final section.  

\newpage
\section{The main result}

We first introduce Theta Characteristics the main Proposition \ref{main} of the present article rely on. We then outline this Proposition, and a somehow local equivalent; namely that locally a nice curve has either rational 2-torsion or a rational Theta Characteristic. Finally, we show a link with Ceresa classes.

\subsection{Theta Characteristics}

References about Theta characteristics include \cite{Mumford1} for a complete and classical survey and \cite{planeTC} focusing on smooth plane quartics.\\

In what follows, let $k$ be a field with $char(k) \neq 2$ and let $C$ be a "nice" curve \footnote{A "nice" curve means a smooth projective and geometrically integral curve in the present article.} on $k$ of genus $g \geqslant 1$. Fix a canonical divisor $K_0$ on $C$ such that $K_0$ is defined over $k$.

\begin{de}[Theta Characteristic] \mbox{}\\
    A Theta Characteristic over k is a divisor $\beta \in \DivOp_C(k)$ such that
    $$ 2 \beta \sim K_0.$$
    We define $\Delta(k) \subset \Pic_C (k) $ to be the set of the \normalfont{classes} \textit{of TC over $k$.\\
    An Odd Theta Characteristic is a Theta Characteristic $\beta$ such that $h^0(\beta):=\dim ~ \Hop^0(\beta)$ is odd. Being odd only depends on the class $[\beta]$ of $\beta \in \Pic_C$, and therefore we set $\OTC(k) \subset \Delta(k)$ to be the} \normalfont{classes} \textit{of OTC over $k$.}
\end{de}

We can observe from the definitions that the difference in $\Pic_C$ two Theta Characteristics gives a $2$-torsion point. This trivial fact will be crucial as it allows us to construct a cocycle which will be in focus in the present article. It also highlights that classes of Theta Characteristics form a finite set of known cardinal.

\begin{prop}
On an algebraic closed field $\kb$,
$$ |\Delta(\kb)| = 2^{2g} \hspace{10mm} and \hspace{10mm} |\OTC(\kb)| = 2^{(g-1)}(2^g -1).$$
\end{prop}

\begin{proof}
    See \cite{Mumford1}
\end{proof}

Since the Galois action on $\Pic_C$ preserves the sets $\Delta$ and $\OTC$, they actually are both finite non-empty Galois sets.

\subsection{Main proposition}

Everything relies on the following fact.

\begin{prop}\label{main} \mbox{}\\
	Let $C$ be a nice (smooth projective and geometrically integral) curve defined over a number field $\K$, and let $J$ be its Jacobian.\\ 
	Suppose $C$ has a $\K$-rational divisor class $P_0$ of degree $1$ (e.g. a rational point).\\
	Then, \emph{at least one} of the following statements is true:
    \begin{enumerate}[i.]
        \item The Mordell-Weil rank $\rank_\K(J)$ of the Jacobian $J$ of $C$ is strictly positive.
        \item The rational 2-torsion $J[2](\K)$ is non-trivial.
        \item There exists a $\K$-rational theta characteristic on the curve.
    \end{enumerate}
\end{prop}

We give two proofs of this result.

\begin{proof}[First proof]
Let $K_0$ be the canonical class. It is a $\K$-rational class of degree $2g-2$, where $g$ is the genus of $C$. The degree $g-1$ class $D_0:=(g-1)P_0$ is $\K$-rational given the assumptions. Then $D:=K_0-2D_0$ is a $\K$-rational class of degree 0, so it corresponds to a point of $J(\K)$ under the association $J \simeq_\K \Pic_C^0$ as allowed by the existence of $P_0$ (see \cite{PoonenShaefer}[Proposition 3.2]). The class $D$ is trivial in $J(\K)/2J(\K)$ if and only if there exists a $\K$-rational class $k_0$ such that $2k_0=K_0$ in $J$; this is exactly the definition of a rational theta characteristic. Now, let us assume that there is no rational theta characteristic (i.e. that statement $iii.$ is not satisfied), then $D$ defines a non-trivial element of $J(\K)/2J(\K) $. Further assume that $J[2](\K) = 0$ is trivial (i.e. $ii.$ is not satisfied either), then it does contribute to $J(\K)/2J(\K)$, so the existence of $0 \neq D \in J(\K)/2J(\K)$ shows that $\rank_\K(J) > 0$.
\end{proof}

\begin{proof}[Second proof]
We can also similarly prove this fact by constructing directly a cocyle attached to $D$. Let $\beta$ be any (class of) theta characteristic. Let 
\begin{align*}
    \zeta : G_\K & \ra J[2]\\
    \sigma & \mapsto  \beta^\sigma -\beta
\end{align*}
We observe that $\zeta$ is well-defined, as the difference of two theta characteristics is always a 2-torsion point; and moreover that $\zeta \in \Hop^1(\K,J[2])$. The cocycle $\zeta$ is trivial if and only if there exists a 2-torsion point $P \in J[2]$ such that $\zeta(\sigma) = P^\sigma -P$ for all Galois automorphisms $\sigma$, i.e. if and only if $(\beta + P)^\sigma = (\beta +P) $ for all $\sigma$. Since the set $\beta + J[2]$ is exactly the set of theta characteristics, we have that $\zeta$ is a coboundary if and only if there exists a $\K$-rational theta characteristic. Now, let us assume there is no rational theta characteristic (or that the statement $iii.$ is not satisfied), then $\zeta$ is a non-trivial cocycle of $\Hop^1(\K,J[2])$. In fact, it belongs to the 2-Selmer group of $J$. Indeed, $\zeta$ is trivial in $\Hop^1(\K,J)$ as is more apparent in the writing $\zeta(\sigma) = (\beta -D_0)^\sigma - (\beta -D_0)$ where $D_0:=(g-1)P_0$ is a $\K$-rational degree $g-1$ divisor class. Yet, from the cohomology of the exact sequence of Galois modules $0 \ra J[2] \ra J \overset{2}{\ra} J \ra 0$, the kernel of the natural map $\Hop^1(\K,J[2]) \ra \Hop^1(\K,J)$ is $J(\K)/2J(\K)$ and this latter naturally injects into the Selmer group. So, in conclusion, $\zeta$ is a non-trivial element of $J(\K)/2J(\K) \subset \Sel_\K(J) \subset \Hop^1(\K,J[2])$. Further assume that $J[2](\K) = 0$ is trivial (i.e. $ii.$ is not satisfied either), then it does contribute to $J(\K)/2J(\K)$, so the existence of $0 \neq \zeta \in J(\K)/2J(\K)$ shows that $\rank_\K(J) > 0$.
\end{proof}

\begin{rem}
    Notice that this result is not interesting for elliptic curves. Indeed, the canonical divisor of elliptic curves is trivial and Theta Characteristics therefore confuse with 2-torsion. Consequently, there is always a rational Theta Characteristic (the zero point).\\
\end{rem}

The point of this result is that in general, it is not so easy to ensure that the Jacobian $J$ of a curve $C$ has a strictly positive rank without exhibiting an explicit rational point on the Jacobian and proving its order is infinite by direct computations. It can be hard to undertake in practice. So if we can check that there are neither rational non-trivial 2-torsion nor rational theta characteristics (but still a rational degree 1 class) on a curve, then we are ensured that the Mordell-Weil rank of the Jacobian $J$ is at least 1. In that case, note that the first proof is constructive as it provides that $K_0-2D_0$ is indeed of infinite order for any rational divisor $D_0$ of degree $g-1$. However, we do not need to exhibit any explicit $D_0$ to ensure positive rank.\\
Unfortunately, even if we only require the existence of a degree $g-1$ rational divisor $D_0$ to construct such an infinite order point on the Jacobian, we are still bound to the existence of a rational degree 1 divisor. It certifies that the isomorphism $J \simeq \Pic^0_C$ behaves correctly under Galois (see \cite{PoonenShaefer}[Proposition 3.2]). This discussion leads to the following remark.\\

\begin{rem}\label{remHypothesis}
    We can replace in \ref{main} and \ref{maincor} the assumption of the existence of a rational degree 1 class by the assumption that the curve is everywhere locally solvable (ELS) and has a rational degree $g-1$ class. Indeed, under these hypotheses the isomorphism $J \simeq \Pic^0_C$ is Galois equivariant (see \cite{PoonenShaefer}[Proposition 3.3]) and we can still construct $D_0$. The proofs therefore work the same.\\
    Actually, we can even relax the hypothesis by replacing the ELS assumption by the assumption that the local period of the curve is 1 everywhere (see \cite{PoonenShaefer}[Proposition 3.3]).\\
\end{rem}

\underline{Example}:\\
Consider a genus 2 curve as the smooth projective normalisation of the affine curve $y^2=f(x)$ with $f$ a squarefree monic polynomial of degree 6 defined over $\Q$. It has two rational points $\infty^+$ and $\infty^-$ at infinity. Since $\infty^+ + \infty^-$ is a canonical divisor, by taking $e:=\infty^-$, Proposition \ref{main} checks for instance whether $[\infty^+ - \infty^-]\in J(\Q)$ is torsion.\\

Furthermore, the proof brings to light the following cocycle, which may be interesting in its own right.
\begin{cor}
    Let $C$ be a nice curve defined over a number field $\K$ that possesses a rational degree 1 class, let $J$ be its Jacobian, and pick a theta characteristic $\beta$.\\
    Then the map $\xi_{TC} : G_\K \ni \sigma \mapsto \beta^\sigma -\beta \in J[2]$ defines an element of $\Hop^1(\K,J[2])$, which does not depend on the choice of $\beta$. This element is trivial iff $C$ admits a $\K$-rational theta characteristic. Moreover, $\xi_{TC} \in \Sel(J)$ belongs to the $2$-Selmer group of $J$ and its image in $\Sha^1(\K,J)[2]$ is trivial.
\end{cor}

\begin{rem}
    This cocycle is the class of the torsor under A[2] parametrizing the Theta Characteristics. See \cite{PoonenXiTC} for another approach.
\end{rem}

\subsection{A local version}
We also prove a local equivalent to the main proposition \ref{main}.

\begin{prop}
    Let $k$ be a number field and $k_v$ its completion at a place $v \nmid 2$. Let $C$ be a nice curve defined over $k_v$ with a rational degree 1 class and $J$ its Jacobian.\\
    Then, \emph{at least one} of the following statements is true:
    \begin{enumerate}[i.]
        \item The rational 2-torsion $J[2](k_v)$ is non-trivial.
        \item There exists a $k_v$-rational theta characteristic on the curve.
    \end{enumerate}
\end{prop}

\begin{proof}
    %Since $C$ is a nice curve on a local field, it possesses a rational degree 1 divisor class. 
    Assume there is no rational theta characteristic (statement $ii.$ is false). According to the last Corollary, the cocycle $\xi_{TC}$ defines a non-trivial element of $J(k_v)/2J(k_v)$. From \cite{PoonenShaefer}[Proposition 12.10], this latter is a $\F_2$-vector space of dimension $\dim_{\F_2}(J[2](k_v))$ because $v\nmid 2$. Therefore the rational 2-torsion is non-trivial (statement $i.$ is true).
\end{proof}

Note that the last proposition is interesting only over $k_p$ for a prime $p \nmid 2$ of bad reduction as there always is a rational Theta Characteristic for a prime of good reduction (see \cite{ratTC}[section 5. remark 2.]).

\subsection{ A link with Ceresa classes}

When a "nice" curve $C$ over $\Q$ of genus $g>1$ has a Ceresa cycle $\kappa_e$ associated to a rational degree 1 class $e$ such that $\kappa_e$ is torsion, so is the class $D:=K_C-(2g-2)e$ (see \cite{Ceresa} lemma 2.10). It implies that $\xi_{TC}$ is torsion as well since it is the image of $D$ under the Kummer map $J(\K)/2J(\K) \hra \Hop^1(\K,J[2])$ (see the first proof of the Proposition $\ref{main}$). This is interesting if the order of torsion of $\kappa_e$ is odd as it induces the triviality of $\xi_{TC}$. The author therefore wonders if there is a more direct link between the Galois cocycle associated with the Ceresa class $\kappa_e$ and $\xi_{TC}$.\\

We may also note the following related Corollary.
\begin{cor}\label{Ceresa}
    Let $C$ be a nice curve over $\Q$ of genus $g>1$, and let $J$ be its Jacobian. Assume $C$ possesses a degree 1 rational class. Assume as well that $C$ admits no rational theta characteristic and no rational non-trivial 2-torsion point. Then, the Jacobian $J$ has a strictly positive Mordell-Weil rank and $C$ has no torsion Ceresa cycle.
\end{cor}

\begin{proof}
    The first statement derives from the Proposition \ref{main}.\\
    For the second, we show the contrapositive. According to \cite{Ceresa}, the existence of a torsion Ceresa cycle $\kappa_e$ associated to a degree 1 rational class $e$ implies that the rational class $D:=K_C-(2g-2)e$, where $K_C$ is the canonical class, is torsion too. If its order of torsion is even, then $C$ has a nonzero rational 2-torsion point $D\in J[2](\Q )$. If its order of torsion $m=2k+1$ is odd, then $(g-1)e-kD$ is a rational theta characteristic on $C$.
\end{proof}

\section{Bypassing Theta Characteristics}

Suppose now that $C$ has genus $g > 1$. Here is a refinement of Proposition \ref{main} that only requires us to pay attention to $J[2]$, so we no longer have to keep an eye on theta characteristics. 

\begin{theo}\label{maincor}
Let $C$ be a nice curve over a number field $\K$ of genus $g>1$ with a rational class of degree $1$ and let $J$ be its Jacobian.\\
If Galois acts transitively over $J[2] \setminus \{ 0 \}$, then
$\rank_\K(J) \geq 1$.
\end{theo}

\begin{proof}
    In order to use Proposition \ref{main}, let us show that under this hypothesis the curve $C$ cannot admit a rational theta characteristic. The existence of a $\K$-rational theta characteristic $\beta_0$ would induce an isomorphism of Galois sets 
    \begin{align*}
        \Delta & \simeq J[2]\\
        \beta & \mapsto  \beta -\beta_0.
    \end{align*}
    But the Galois action over $\Delta$ cannot be transitive, as $\Delta$ naturally splits between even and odd theta characteristics. So in the case where the \'etale algebra of $J[2] \setminus \{0 \}$ is actually a field (i.e. when the Galois action on $J[2] \setminus \{ 0 \}$ is transitive), then obviously there is no rational 2-torsion point in $J$. There is strictly more than one Theta Characteristic of both types if and only if the genus is greater than 1. In this situation, there is no rational theta characteristic either, otherwise there would be at least two distinct Galois orbits (one for the odd theta characteristics, and one for the even ones). This concludes the proof.
\end{proof}

\begin{rem}
    We can slightly soften the hypothesis. We refer to the Remark \ref{remHypothesis}.\\
\end{rem}

If Galois representations afforded by the 2-torsion of Jacobians are less studied than its elliptic curve equivalent, we still expect the image Galois group to be $\Sp$ in the generic case and therefore to act transitively over $J[2]\setminus\{0\}$. Thus, according to the last theorem, it is reasonable to think that "most" of the nice curves of genus $g>1$ that possesses a rational class of degree 1 have positive rank. This would be in stark contrast with the case $g=1$, as a famous conjecture states that among the curves of genus $1$ defined over $\Q$ and admitting a rational point (i.e. an elliptic curve), $50\%$ have rank $r=1$ and $50\%$ have rank $r=0$.\\

A smooth plane quintic naturally possesses a rational divisor of degree $5=g-1$. For the reasons evoked in the Remark \ref{remHypothesis}, a generic smooth plane quintic (with therefore transitive Galois action on the 2-torsion) which is also ELS would have positive rank.

\section{Computational Approach}

Checking whether a rational 2-torsion point and/or a rational theta characteristic exist then arises as a natural task. Denote by $\Delta$ the set of (classes of) theta characteristics. In practice and in the case where we know a rational point on the curve, we can use algorithms to compute the Galois representations attached to $J[2]$ and to the $\F_2$-permutation module $\F_2^\Delta$ of the Theta Characteristics (implemented by Mascot for the former \cite{Mascot}, and by the author for the latter (to appear soon)). These algorithms compute the \'etale algebras of $J[2] \setminus \{0\}$ and $\Delta$ as quotient algebras $\Q[y]/\chi(y)$ for some squarefree $\chi(y) \in \Q[y]$. For either of them, the underlying Galois set has a rational element if and only if the associated $\chi$ has a linear factor.\\

In a more computational point of view, the Theorem \ref{maincor} can be re-expressed in the following way.

\begin{cor} \label{CorIrreducible}
Let $C$ be a nice curve over $\K$ with a rational class of degree $1$ and let $J$ be its Jacobian. Let $\chi$ be a polynomial such that the \'etale algebra of the Galois set $J[2]\setminus 0$ is isomorphic to $\K[y]/\chi(y)$.\\
If $\chi$ is irreducible, then
$\rank_\K(J) \geq 1$.
\end{cor}

The point is that when $\K=\Q$, Mascot's algorithm \cite{Mascot} (GitHub link \cite{Github}) can compute such a polynomial $\chi$. This algorithm has been implemented in the PARI/GP language \cite{PARI2}.

\section{The positive rank trick in families of curves}

We now extend our result to families of algebraic curves. See \cite{HilbertIrreducibility}[chapter 3.] for a reference on thin sets.

\begin{prop}\label{family} \mbox{}\\
    Let $X$ be a nice curve of genus $g>1$ defined over $\Q(t)$ with a $\Q(t)$-rational class of degree $1$. Assume that there exists $b\in \Q$ such that the fiber $X_b$ defined over $\Q$ is smooth, and that Galois acts transitively on its nonzero 2-torsion $J_b[2]\setminus\{ 0 \}$.\\
    Then, there exists a thin set $Z \subset \Q$ such that for every $a\in \Q\setminus Z$, the fiber $X_a$ is smooth and has strictly positive Mordell-Weil rank.
\end{prop}

\begin{proof}
    Loosely speaking, we will show that the assumption on $X_b$ ensures that the Galois action is transitive over $\Jmoins$, where $J/\Q(t)$ is the Jacobian of $X$. We will then conclude by \ref{CorIrreducible} and Hilbert's irreducibility theorem.\\
    
    Formally, the curve $X$ is defined over $\Q(t)$ and, for every $a \in \Q$, we note $X_a$ the fiber at $a\in \Q$ defined over $\Q$. The curve $X$ has a $\Q(t)$-rational degree 1 class by assumption. Denote by $J$ the Jacobian of $X$, which is defined over $\Q(t)$ as well. Let $\Lc:=\Q(t)(J[2])$ be the field of definition of the 2-torsion over $\Q(t)$. It is a finite Galois extension and we write $G:=\Gal(\Lc/\Q(t))$ for its Galois group.
    
    Let $J_b$ be the Jacobian of $X_b$. Pick $\pi_b \in \Q(J_b)$ in the function field of $J_b$ which is defined and injective on $J_b[2]$, and lift it arbitrarily to $\pi \in \Q(t)(J)$. Then $\pi$ restricts to an injective and Galois-equivariant map $\pi: J[2] \ra \Lc$. From $\pi$, we construct the polynomial $\chi(x) := \prod_{P\in J[2]\setminus\{0\}} (x-\pi(P))$. Since $\pi$ is Galois-equivariant, $\chi$ is invariant under Galois, whence $\chi\in \Q(t)[x]$. The injectivity of $\pi$ ensures that $\chi$ is separable and of Galois group $G$ over $\Q(t)$. In fact, because $\pi$ is an isomorphism of Galois sets between $\Jmoins$ and $\Delta := \{\text{roots of } \chi \} $, it induces an isomorphism of permutation representations between $\Gal_{\Q(t)}(\chi) \ra \Sym(\Delta)$ and $G \ra \Sym(\Jmoins)$. In particular, this proves that $G=\Gal_{\Q(t)}(\chi)$ indeed. We are going to show that $G$ is transitive, using the assumption on $X_b$. So we now study the specialisations of $\chi$.\\

    Let $d(t) \in \Q[t]$ be a common denominator for $\chi$, and let $\delta(t) \in \Q(t)$ be the discriminant of $\chi$. Then the set
    \[ Z_1:= \{a\in \Q ~ | ~ d(a) = 0 \text{ or } \delta(a)  = 0\} \]
    is finite, and for $a \in \Q \setminus Z_1$, the specialisation $\chi_a$ of $\chi$ at $t=a$ is a well-defined separable polynomial in $\Q[x]$. We write $G_a:= \Gal_\Q(\chi_a)$ for its Galois group. Let also
    \[ Z_2:= \{a\in \Q ~ | ~ X \text{ has bad reduction at } t=a \}; \]
    then $Z_2$ is also finite since we assumed that $X$ has at least one good fiber.
    
    For any $a\in \Q$, we note $\ev_a $ the specialisation at $t=a$. The morphism $\ev^J_a:J[2] \ra J_a[2]$ is surjective if $a$ is of good reduction. Therefore, for every $a\in\Q \setminus (Z_1\cup Z_2)$, there exists a unique map $\pi_a : J_a[2] \ra \Qb$ verifying $\pi_a \circ \ev^J_a = \ev_a \circ \pi$. It is injective and Galois-equivariant. We set $\Delta_a:=\{\text{roots of } \chi_a \}$. Similarly to the global case, since $\chi_a(x) = \prod_{P\in J_a[2] \moins } (x - \pi_a(P))$, the morphism of Galois sets $\pi_a$ realises an isomorphism of Galois representations between $G_a \ra \Sym(\Delta_a)$ and $\Gal(J_a[2]/\Q) \ra \Sym(J_a[2] \moins)$ for every $a\in \Q \setminus (Z_1 \cup Z_2)$. In fact, for such $a$, the \'etale algebra of the nonzero $2$-torsion of the good fiber $X_a$ is isomorphic to $\Q[x]/\chi_a(x)$.\\
    
    By assumption, $b \not \in Z_1 \cup Z_2$, and as we just showed, the polynomial $\chi_b(x) \in \Q[x]$ then describes the Galois action on $J_b[2] \setminus \{0\}$. The transitivity of the Galois action at $t=b$ therefore ensures that $\chi_b$ is irreducible. The polynomial $\chi$ is also irreducible as one of its specialisation is. It implies that $G$ is transitive thanks to the isomorphism of Galois sets $\Delta \simeq J[2] \setminus \{0\}$. Moreover, by the Hilbert irreducibility theorem, there exists a thin set $Z_3 \subset \Q$ such that for every $a\in \Q\setminus Z_3$, the polynomial $\chi_a$ is irreducible. The set $Z:= Z_1\cup Z_2 \cup Z_3$ is still thin. Finally, for every $a\in \Q\setminus Z$, the curve $X_a$ is nice, admits a transitive Galois action over its nonzero 2-torsion, and a $\Q$-rational class of degree 1 obtained by specialising that of $X$. By Corollary \ref{CorIrreducible}, we conclude that $J_a$ has positive Mordell-Weil rank for every $a \in \Q\setminus Z$.
\end{proof}

We can therefore quite easily produce families of curves with infinitely many of them of positive rank. We give explicit examples in the next section.\\

Notice, as a consequence of \ref{Ceresa}, that every fiber $X_a$ for $a\in\Q\setminus Z$ defines a curve over $\Q$ with no torsion Ceresa cycle. This therefore produces a tool to construct explicit family of curves with non-vanishing Ceresa cycles. Another explicit example of a family of curves with infinite order Ceresa class is given in \cite{NonVanishingCeresa}[Corollary 1.2].

\newpage

\section{Examples}

%\subsection{Generic Picard curves}

%Let $C/\Q$ be a Picard curve $y^3=f(x)$ with $f$ a polynomial over $\Q$ of degree 4. The curve $C$ is of genus $3$ and therefore its canonical divisor defines a degree $4$ rational divisor. Picard curves defined over $\Q$ are famously trigonal over $\Q$. A linear system of degree 3 induces a rational degree $3$ divisor. In particular, $C$ possesses a rational divisor of degree $4 \wedge 3 =1$. Studying the 2 Galois representation $\rho$ of a generic Picard curve may allow us to show that it has positive rank using \ref{maincor}. In \cite{Picard}[Appendix], Sutherland identifies the maximal allowed image $G$ of $\rho$ as embedded in $\Aut(J[2])\simeq\Sp(6,\F_2)$. This permits us to compute the Galois orbits of $J[2]$ of such a group by direct computation using \textit{MAGMA}.\revoir{Since a generic Picard group admits $G$ as image of $\rho$,} we conclude that a Picard curve has rational TC.

\subsection{A transitive action on the 2-torsion}

Consider the projective smooth curve $C_1$ with affine model the plane quartic defined over $\Q$ by 
$$ -x^3 + (y^2 - 2)x^2 + (-y^3 - 1)x + (y^2 + y) =0.$$
It is smooth of genus $g=3$, and has an obvious rational point $P_0:= (0,0)$. Denote by $J_1:=\Jac(C_1)$ its Jacobian. The Galois set $J_1[2]\setminus \{0\}$ admits an \'etale algebra realised by
\begin{gather*}
\chi_1(x) =
\parbox[t]{0.85\displaywidth }{\raggedright{\small
$
x^{63} - 64x^{62} - 864x^{61} + 99088x^{60} + 3147824x^{59} + 38645104x^{58} + 220851800x^{57} + 436655752x^{56}+ 643508880x^{55} + 38251581568x^{54} + 490860398080x^{53} + 3307604786624x^{52} + 14969666840812x^{51} + 51267708232424x^{50} + 146362375589644x^{49 } + 377890207690192x^{48 } + 906115797249104x^{47 } + 1958095383272224x^{46 } + 3765550593270872x^{45 } + 6918434059596824x^{44 } + 13265441301020808x^{43 } + 24473394958608600x^{42 } + 35658387247444144x^{41 } + 39428900697762272x^{40 } + 62242353754664918x^{39 } + 167417473214683624x^{38 } + 311338727446040676x^{37 } + 268023809404624792x^{36 } - 5764482841567130x^{35 } - 112184000527819920x^{34 } + 166883546405206664x^{33 } + 460583434543199080x^{32 } + 538227233748134224x^{31 } + 480849233193005120x^{30 } + 100484087630669160x^{29} - 315259092750715144x^{28} + 21600350738000124x^{27 } + 517511161339352888x^{26 } + 233689365450397204x^{25} - 51875634407421504x^{24 } + 177019364511448292x^{23 } + 115783645865255736x^{22} - 160055044846120332x^{21} - 32043325719462600x^{20 } + 161026888829507896x^{19 } + 64048643293573304x^{18 } - 51017236515333304x^{17 } - 24873586109487464x^{16} + 18821385350026217x^{15 } + 12441207943213152x^{14 } - 4132173650718804x^{13 } - 3648509134108696x^{12 } + 919811261539878x^{11 } + 730386886880200x^{10 } - 175242998184516x^{9 } - 91963999426888x^{8 } + 27821164721329x^{7 } + 8871762965632x^{6 } - 2538051208864x^{5 } - 584354054912x^{4 } + 148601942272x^{3 } + 43158978560x^{2 } + 2803515392x + 27541504
$
}}
\end{gather*}

as calculated by Mascot's algorithm \cite{Mascot}. This polynomial can be checked to be irreducible over $\Q$. According to Corollary \ref{CorIrreducible}, we therefore have
$$ \rank_\Q(J_1) \geq 1.$$

\subsection{A non-transitive action on the 2-torsion}

Consider the projective smooth curve $C_2$ with affine model the plane quartic defined over $\Q$ by
$$ x^4 + x^3 + 7x -y^4 + 1=0.$$
It has a trivial rational point $P_0 := [1,1,0]\in \Proj^1$ and is of genus $g=3$. Denote $J_2:=\Jac(C_2)$ its Jacobian. The Galois orbits decomposition of $J_2[2]\setminus\{0\}$ obtained by factoring the polynomial computed by Mascot's algorithm \cite{Mascot} is
$$ J_2[2]\setminus\{0\}: ~ 3_1 ~ 12_1 ~ 48_1. $$
We therefore cannot ensure positive rank yet. We have computed the \'etale algebras of the odd and even theta characteristics denoted respectively by $\OTC$ and $\Delta_{Even}$. Assuming the correctness of the algorithm of the author, the Galois orbits decompositions are
$$ \OTC: ~ 4_1 ~ 24_1 \hspace{5mm} \text{and} \hspace{5mm} \Delta_{Even}: ~ 12_1 ~ 24_1.$$
There is therefore no rational non-trivial 2-torsion point and no rational theta characteristic. According to the Proposition \ref{main}, we therefore have
$$ \rank_\Q(J_2) \geq 1.$$

\subsection{A family of curves with infinitely many fibers of positive rank}

In \cite{NicolasFamily}[section 4.3], Mascot computed a polynomial $\chi \in \Q(t)[x]$ expressing the etale algebra of the 2-torsion (see the section Computational Approach of the present article) of the smooth plane quartic defined over $\Q(t)$ by the following equation:
$$ x^4 + (2 - t)y^4 + 2x^3 + x(x + y) + (t - 1)(y + x^2 + x) = 0. $$
Every curve in the family has the rational point $(0,0)$. The author ensures that the computed polynomial is irreducible over $\Q(t)$ and hence we are in the situation of the proof of \ref{family}. Therefore, infinitely many curves in the family have positive rank.

\subsection{A constructed family of curves with infinitely many fibers of positive rank}

We now construct a family of curves with infinitely many fibers of positive rank. Consider the family of projective curves $X$ defined over $\Q(t)$ by the affine model
$$ -x^3 + (y^2 - 2)x^2 + (-y^3 - 1)x + (y^2 + y)(t^2-t+1) =0.$$
The curve $X$ is smooth and possesses the rational point $P_0=(0,0)$ and is of genus $g=3$. The fiber $X_b$ for $b=0$ is isomorphic to $C_1$ and therefore admits a transitive Galois action over its non-trivial 2-torsion. According to Proposition \ref{family}, this implies that there exists a thin set $Z\subset \Q$ such that for every $a\in \Q \setminus Z$, the projective curve $X_a$ defined over $\Q$ by the affine model
$$ -x^3 + (y^2 - 2)x^2 + (-y^3 - 1)x + (y^2 + y)(a^2-a+1) =0$$
is smooth and its Jacobian $J_a$ satisfies 
$$\rank_\Q(J_a) \geq 1.$$

Building over $C_1$ or any other curve verifying the Corollary \ref{CorIrreducible}, it is not very complicated to provide such families of curves. For instance, we could also have exhibited the smooth families 
$$ (t^3-1)x^3 + (y^2 - 2(t+1))x^2 + (t^4-2t^3+1)(-y^3 - 1)x + (t^2+1)(y^2 + y) =0$$
with rational point $(0,0)$ and fiber at $t=0$ isomorphic to $C_1$. Or
$$-x^3 + (y^2 - 2)x^2 + (-y^3 - 1)x + ((y^2 + y) -(t^2+t)) =0.$$
with rational point $(0,t)$ and fiber at $t=0$ isomorphic to $C_1$.
\newpage
\bibliography{bib}

\end{document}

%% file: MetaDonnees.tex
\usepackage{authblk}
\usepackage[dvipsnames]{xcolor}
\usepackage[12pt]{extsizes}
\usepackage[utf8]{inputenc}
\usepackage[english]{babel}
\usepackage{fancybox}		
\usepackage{amsmath}
\usepackage{amsthm}
\usepackage{comment}
\usepackage{amsfonts}
\usepackage{enumerate}
\usepackage[shortlabels]{enumitem}
\usepackage{graphicx}
\usepackage[left=2cm, right =2cm, top = 2cm, bottom =2cm]{geometry}
\usepackage{fullpage}
\usepackage{hyperref}
\usepackage{xcolor}
\usepackage{minted}
\usepackage{stmaryrd} 

\usepackage{tikz}
\usepackage{tikz-qtree}
\usepackage{tikz-cd}
\usepackage{graphicx}
\usepackage{array}
\usepackage{mathabx}
\usepackage{amssymb}

\usepackage{mathtools}
\makeindex

\newcommand{\Lc}{\mathbb{L}}
\newcommand{\K}{\mathbb{K}}
\newcommand{\Q}{\mathbb{Q}}
\newcommand{\Proj}{\mathbb{P}}
\newcommand{\Qb}{\overline{\Q}}

\newcommand{\Z}{\mathbb{Z}}

\newcommand{\F}{\mathbb{F}}

\newcommand{\Sel}{\Selmer^{(2)}}

\newcommand{\OTC}{\Delta_{Odd}}

\newcommand{\kb}{\overline{k}}

\newcommand{\Sp}{Sp_{2g}(\F_2)}

\newcommand{\ra}{\rightarrow}
\newcommand{\hra}{\hookrightarrow}

\usepackage[OT2,T1]{fontenc}
\DeclareSymbolFont{cyrletters}{OT2}{wncyr}{m}{n}
\DeclareMathSymbol{\Sha}{\mathalpha}{cyrletters}{"58}

\newtheorem{de}{Definition}
\newtheorem{theo}{Theorem}
\newtheorem{prop}[theo]{Proposition}
 
\newtheorem{cor}[theo]{Corollary}

\newcommand{\Pic}{\operatorname{Pic}}
\newcommand{\Jac}{\operatorname{Jac}}
\newcommand{\Selmer}{\operatorname{Sel}}
\newcommand{\Hop}{\operatorname{H}}

\newcommand{\DivOp}{\operatorname{Div}}
\newcommand{\rank}{\operatorname{rank}}

\renewcommand{\contentsname}{Sommaire} 
\everymath{\displaystyle}
\usepackage{url}